\documentclass[reqno, a4paper, 12pt]{amsart}
\usepackage[utf8]{inputenc}
\usepackage[T1]{fontenc}
\usepackage[english]{babel}
\usepackage[babel]{microtype}
\usepackage[a4paper, margin=1.2in]{geometry}
\usepackage{amsmath, amssymb, amsthm}
\usepackage[nobysame]{amsrefs}
\usepackage{mathrsfs,mathtools}
\usepackage{bbm}
\usepackage{dsfont}
\usepackage{lmodern}
\usepackage{cases}
\usepackage{enumitem}
\usepackage{centernot}
\usepackage{todonotes}
\usepackage{graphicx}
\usepackage{float}
\usepackage{hyperref}

\DeclareFontShape{OMX}{cmex}{m}{n}{%
	<-7.5> cmex7
	<7.5-8.5> cmex8
	<8.5-9.5> cmex9
	<9.5-> cmex10
}{}
\DeclareSymbolFont{largesymbols}{OMX}{cmex}{m}{n} 

\numberwithin{equation}{section}
\newtheorem{thm}    {Theorem}
\newtheorem{lem}      [thm] {Lemma}

\newtheorem{dfn} [thm] {Definition}
\newtheorem{pps}[thm] {Proposition}
\newtheorem{cor}  [thm] {Corollary}

\newtheorem{exm}    [thm] {Example}

\setlist[enumerate,1]{label=(\roman*)}

\newcommand{\ds}{\displaystyle}
\newtheorem{THEO}{Theorem}

\begin{document}
	
\title[A refined Gallai-Edmonds structure theorem]{A refined Gallai-Edmonds structure theorem for weighted matching polynomials}
	
\author[T.~J.~Spier]{Thomás Jung Spier \\ \today}
\address{IMPA, Rio de Janeiro, RJ, Brasil}
\email{thomasjs@impa.br}
	
	
\begin{abstract}
In this work, we prove a refinement of the Gallai-Edmonds structure theorem for weighted matching polynomials by Ku and Wong. Our proof uses a connection between matching polynomials and branched continued fractions. We also show how this is related to a modification by Sylvester of the classical Sturm's theorem on the number of zeros of a real polynomial in an interval. In addition, we obtain some other results about zeros of matching polynomials.
\end{abstract}


\makeatletter
\@namedef{subjclassname@2020}{\textup{2020} Mathematics Subject Classification}
\makeatother
	
	
\clearpage\maketitle
\thispagestyle{empty}
	
	
\section{Introduction}\label{intro}

Let $G$ be a finite simple graph. A {\it matching} in $G$ is a set of pairwise non-adjacent edges. The celebrated Gallai-Edmonds theorem~\cite{gallai1963kritische,edmonds1965paths} gives the structure of maximum matchings in a graph.

A vertex $v$ is covered by the matching $M$ if there is an edge in $M$ which is incident to $v$. The vertex $v$ is {\it essential} if
there is a maximum size matching in $G$ which leaves $v$ uncovered. If all the vertices of a graph are essential then the graph is {\it factor-critical}.

\begin{THEO}(Gallai's lemma~\cite{gallai1963kritische})\label{gallai.old} If $G$ is connected and factor-critical then each maximum size matching leaves exactly one vertex uncovered.	
\end{THEO}

Denote by $D_G$ the set of essential vertices of $G$. Write $A_G$ for the frontier of $D_G$, i.e. the set of vertices which are not in $D_G$ but have a neighbor in $D_G$, and define $C_G$ as $V(G)\setminus (D_G\sqcup A_G)$. 

\begin{THEO}(Gallai-Edmonds structure theorem~\cite{gallai1963kritische,edmonds1965paths})\label{gallai-edmonds} For every graph $G$ it holds:
\begin{enumerate}[label=(\alph*)]
	\item the components of the subgraph induced by $D_G$ are factor-critical;
	\item the subgraph induced by $C_G$ has a perfect matching;
	\item let $S$ be a nonempty subset of $A_G$. Then there are at least $|S|+1$ components of $D_G$ which are connected to a vertex in $S$ in the graph $G$; 
	\item if $M$ is any maximum matching of $G$, it contains a near perfect matching of each component of $D_G$, a perfect matching of each component of $C_G$ and matches all vertices of $A_G$ with vertices in distinct components of $D_G$;
	\item if $def(G)$ is the number of vertices left uncovered by a maximum matching in $G$, then $def(G)=c(D_G)-|A_G|$, where $c(D_G)$ denotes the number of connected components of the graph spanned by $D_G$.
\end{enumerate}
\end{THEO}

Ku and Wong~\cite[p. 3390, Thms. 4.12 and 4.13]{ku2013gallai}, building on the work of Godsil~\cite{godsil1995algebraic} and  Ku and Chen~\cite{ku2010analogue}, generalized the Theorems~\ref{gallai.old} and~\ref{gallai-edmonds} for the context of {\it weighted matching polynomials}. Following a line of investigation pursued by Lovász and Plummer~\cite{lovasz2009matching} Ku and Wong were also able to generalize in the works~\cite{ku2010extensions,ku2013gallai,ku2011generalized,ku2009properties,ku2013generalizing} some other classical concepts of matching theory for weighted matching polynomials. Recently, Bencs and Mészáros~\cite[p. 5-6, Thms. 1.9-1.12]{bencs2020atoms} also proved versions of the theorems by Ku and Wong for infinite and random graphs. 

We briefly recall some facts about weighted matching polynomials in order to state the work of Ku and Wong. Let $G$ be the complete graph with vertex set $[n]$. Define variable weights $x-r_i$ and non-positive weights $\lambda_{jk}$ for each of the vertices and edges, respectively. Denote by $\mathcal{M}_G$ the set of all matchings of $G$ and for simplicity write $i\notin M$ if the vertex $i$ is not covered by the matching $M$. Then the {\it weighted matching polynomial} of $G$ is 

\[\mu(G)(x):=\ds\sum_{M\in\mathcal{M}_G}\ds\prod_{i\not\in M}(x-r_i)\ds\prod_{jk\in M}\lambda_{jk}.\] 
 
The Heilmann-Lieb theorem~\cite[p. 200, Thm. 4.2]{heilmann-lieb} says that all the zeros of $\mu(G)$ are real. Furthermore, it says that for every vertex $i$ the zeros of $\mu(G)$ and $\mu(G\setminus i)$ interlace, i.e. between every two zeros of $\mu(G)$ there is a zero of $\mu(G\setminus i)$ and vice versa. For a real number $\theta$ denote by $m_\theta(G)$ the multiplicity of $\theta$ as a zero of $\mu(G)$. As a consequence of the interlacing of the zeros of $\mu(G)$ and $\mu(G\setminus i)$ it holds that $m_\theta(G\setminus i)$ belongs to $\{m_\theta(G),m_\theta(G)\pm 1\}$. Separate the vertices of $G$ according to:
\begin{itemize}
	\item $i$ is $\theta$-essential if $m_\theta(G\setminus i)=m_\theta(G)-1$;
	\item $i$ is $\theta$-neutral if $m_\theta(G\setminus i)=m_\theta(G)$;
	\item $i$ is $\theta$-positive if $m_\theta(G\setminus i)=m_\theta(G)+1$.
\end{itemize}

With this definition, denote by $D_{\theta,G}$ and $N_{\theta,G}$ the sets of $\theta$-essential and $\theta$-neutral vertices, respectively. Also, denote by $A_{\theta,G}$ the frontier of $D_{\theta,G}$ and by $P_{\theta,G}$ the set of $\theta$-positive vertices which are not in $A_{\theta,G}$. A graph where all vertices are $\theta$-essential is called $\theta$-critical. In this context the analogues of Theorems~\ref{gallai.old} and~\ref{gallai-edmonds} by Ku and Wong are:

\begin{THEO} (The analogue of the Gallai's lemma by Ku and Wong~\cite{ku2013gallai})\label{kuwong.gallai} Let $G$ be a connected $\theta$-critical graph. Then $\theta$ is a simple zero of $\mu(G)$. 
\end{THEO}

\begin{THEO}(The analogue of the Gallai-Edmonds structure theorem by Ku and Wong~\cite{ku2013gallai})\label{kuwong.gallaiedmonds} Let $\theta$ be a zero of $\mu(G)$. Then:
\begin{enumerate}[label=(\alph*)]
	\item the components of the subgraph induced by $D_{\theta,G}$ are $\theta$-critical;
	\item any vertex of the subgraph spanned by $N_{\theta,G}\sqcup P_{\theta,G}$ is not $\theta$-essential in $N_{\theta,G}\sqcup P_{\theta,G}$;
	\item $m_\theta(G)=c(D_{\theta,G})-|A_{\theta,G}|$.
\end{enumerate}
\end{THEO}

As observed by Godsil~\cite[p. 1]{godsil1995algebraic}, if the vertex and edge weights of $G$ are $x$ and $-1$, respectively, then for $\theta$ equal to zero  it holds $m_0(G)=def(G)$ and also $D_{0,G}=D_G$, $A_{0,G}=A_G$ and $N_{0,G}\sqcup P_{0,G}=C_G$. This shows that Theorems~\ref{kuwong.gallai} and~\ref{kuwong.gallaiedmonds} generalize the Theorems~\ref{gallai.old} and~\ref{gallai-edmonds}. The main ingredient in the proof of Theorems~\ref{kuwong.gallai} and~\ref{kuwong.gallaiedmonds} by Ku and Wong is the following result, which appears in~\cite[p. 95, Lem. 3.2.2]{ku2013gallai}. 

\begin{THEO} (Stability theorem by Ku and Wong~\cite{ku2013gallai})\label{kuwong.stability} If $i$ is in $A_{\theta,G}$, then:
\begin{itemize}
	\item $D_{\theta,G\setminus i}=D_{\theta,G}$;
	\item $P_{\theta,G\setminus i}=P_{\theta,G}$;
	\item $N_{\theta,G\setminus i}=N_{\theta,G}$;
	\item $A_{\theta,G\setminus i}=A_{\theta,G}\setminus i$.
\end{itemize}
\end{THEO}

In this work we present refinements of Theorems~\ref{kuwong.gallaiedmonds} and~\ref{kuwong.stability} with a simpler proof. The main novelty is the following new stability lemma which gives more precise information of how a matching polynomial changes when a vertex in $A_{\theta,G}$ is deleted. 

\begin{thm}(Stability lemma)\label{newstability} If $i$ is in $A_{\theta,G}$, then $\dfrac{\mu(G\setminus i)}{\mu(G\setminus \{ i,j\})}(\theta)=\dfrac{\mu(G)}{\mu(G\setminus j)}(\theta)$ for every vertex $j$ different from $i$.
\end{thm}

Our proof uses a connection between matching polynomials and branched continued fractions which was originally observed by Viennot~\cite[p. 149]{viennot1985combinatorial}. For this reason the proof is inspired by results in the theories of continued fractions and orthogonal polynomials.
	
In the course of the proof we obtain a generalization for weighted matching polynomials of the following modification by Sylvester~\cite{sylvester1853lxxi} of the classical Sturm's theorem~\cite{sturm2009memoire} (or see~\cite[p. 305, Thm. 7.10]{khrushchev_orthogonal}) on the number of zeros of a real polynomial in an interval. 

Consider two monic real polynomials $p(x)$ and $q(x)$ of degrees $n$ and $n-1$, respectively, with real and distinct zeros. Assume that the zeros of $p$ and $q$ are different and interlace. In particular, one can take $q$ as the derivative of $p$ divided by $n$. It is known, as can be seen in~\cite[p. 141, Lem. 5.1]{godsil2013algebraic}, that performing the euclidean algorithm for $p$ and $q$ results in:

\[\dfrac{p}{q}(x)=x-r_1+\dfrac{\lambda_1}{x-r_2+\dfrac{\lambda_2}{\dots+\dfrac{\lambda_{n-1}}{x-r_n}}},\]

\noindent where $r_i$ is a real number and $\lambda_i$ is negative for every $i$. The sequence of partial numerators of this continued fraction is known as the Sturm sequence for the pair $(p,q)$ and is the initial segment of an orthogonal polynomial sequence. 

Denote by $\tau_i$ and $\hat{\tau}_i$ the continued fractions

\[\tau_i(x):=x-r_i+\dfrac{\lambda_i}{x-r_{i+1}+\dfrac{\lambda_{i+1}}{\dots+\dfrac{\lambda_{n-1}}{x-r_n}}},\quad \hat{\tau}_i(x):=x-r_i+\dfrac{\lambda_{i-1}}{x-r_{i-1}+\dfrac{\lambda_{i-2}}{\dots+\dfrac{\lambda_1}{x-r_1}}}.\]

For a real number $\theta$ let $V(\theta)$ and $\hat{V}(\theta)$ be the number of positive terms in the sets $\{\tau_1(\theta),\tau_2(\theta),\dots,\tau_n(\theta)\}$ and $\{\hat{\tau}_1(\theta),\hat{\tau}_2(\theta),\dots,\hat{\tau}_n(\theta)\}$, respectively. Note that except for a finite number of values for $\theta$ both sets have only positive and negative elements. 

\begin{THEO} (Sylvester modification of Sturm's theorem~\cite{sturm2009memoire})\label{sylvester} Both $V(\theta)$ and $\hat{V}(\theta)$, when defined, are equal to the number of zeros of $p(x)$ in the interval $(-\infty,\theta)$.
\end{THEO}  

Our version of Theorem~\ref{sylvester} reads as follows. Let $c:i_1\to i_n$ be a hamiltonian path in the graph $G$ and $-c:i_n\to i_1$ be its reverse path. For a real number $\theta$ denote by $V_c(\theta)$ the number of positive terms in: 

\[\left\{\dfrac{\mu(G)}{\mu(G\setminus i_1)}(\theta),\,\dfrac{\mu(G\setminus i_1)}{\mu(G\setminus \{i_1,i_2\})}(\theta),\,\dots,\,\dfrac{\mu(G\setminus \{i_1,\dots,i_{n-1}\})}{\mu(G\setminus \{i_1,\dots,i_{n-1},i_n\})}(\theta)\right\}.\]

Similarly, denote by $V_{-c}(\theta)$ the analogous counting for the reverse path $-c$.

\begin{thm}\label{newsylvester} Both $V_c(\theta)$ and $V_{-c}(\theta)$, when defined, are equal to the number of zeros of $\mu(G)$ in the interval $(-\infty,\theta)$.
\end{thm}

In the next section we define multivariate matching polynomials and present their connection to branched continued fractions. In section~\ref{gallaiedmonds} we prove the Stability Theorem~\ref{newstability} and Sylvester Theorem~\ref{newsylvester}. Finally, in section~\ref{applications} we present some other results about the largest zero of weighted matching polynomials.
	

\section{Graph Continued Fractions}\label{graphcfs}

In this section, following Viennot~\cite[p. 149]{viennot1985combinatorial}, we establish the connection between multivariate matching polynomials and branched continued fractions.

Let $G$ be the complete graph with vertex set $[n]$. Define variable weights $x_i$ and non-positive weights $\lambda_{jk}$ for each of the vertices and edges, respectively. Two vertices $i$ and $j$ are neighbors if $\lambda_{ij}$ is non-zero.

Then the {\it multivariate matching polynomial} of $G$ is

\[\mu(G):=\ds\sum_{M\in\mathcal{M}_G}\ds\prod_{i\not\in M}x_i\ds\prod_{jk\in M}\lambda_{jk}.\]  

This is a real multivariate polynomial in the $n$ vertex variables $x_i$. It is also convenient to define $\mu(\emptyset)=1$. 

The next lemma, which appears in Godsil's book~\cite[p. 2, Thm. 1.1]{godsil2013algebraic}, has some recurrences satisfied by the multivariate matching polynomial.

\begin{lem}\label{recurrencesMatching} Let $G$ and $H$ be weighted graphs and $i$ and $j$ be vertices in $G$. Then,
\begin{enumerate}[label=(\alph*)]
	\item $\mu(G\sqcup H)=\mu(G)\cdot \mu(H)$;
	\item $\mu(G)=x_i\mu(G\setminus i)+\ds\sum_{k\neq i}\lambda_{ik}\mu(G\setminus \{i,k\})$;
	\item $\mu(G)=\lambda_{ij}\mu(G\setminus \{i,j\})+\mu(G\setminus ij)$;
	\item $\partial_i \mu(G)=\mu(G\setminus i)$.
\end{enumerate} 
\end{lem}
\begin{proof} $(a)$ Since the matchings of $G\sqcup H$ are a pair of matchings in $G$ and $H$, the result follows. 
	
$(b)$ The matchings of $G$ that do not cover the vertex $i$ contribute $x_i\mu(G\setminus i)$ to $\mu(G)$. The matchings of $G$ which cover the vertex $i$ must use one of the incident edges and therefore contribute $\ds\sum_{k\neq i}\lambda_{ik}\mu(G\setminus \{i,k\})$ in total to $\mu(G)$. 

$(c)$ If we separate the matchings of $G$ into those that contain the edge $ij$, or not, the conclusion follows. 

$(d)$ This item is an immediate consequence of item $(b)$.
\end{proof}

For every rooted weighted tree, a branched continued fraction can be associated in a natural way, as exemplified in Figure~\ref{branchedcf}. This associated branched continued fraction can be defined recursively as described below.

Denote by $\alpha_i(T)$ the branched continued fraction of the weighted rooted tree $T$ with root $i$. If the vertex $i$ is isolated in $T$, then $\alpha_i(T)$ is, by definition, equal to $x_i$. On the other hand, if $i$ is not isolated in $T$, then we have the recurrence $\alpha_i(T)=x_i+\ds\sum_{j\sim i}\dfrac{\lambda_{ij}}{ \alpha_j (T\setminus i)}$, where the sum is over all neighbors of $i$.

For a rooted weighted tree we call its associated branched continued fraction its \textit{tree continued fraction}. 

\begin{figure}[h]
	\includegraphics[width=\linewidth]{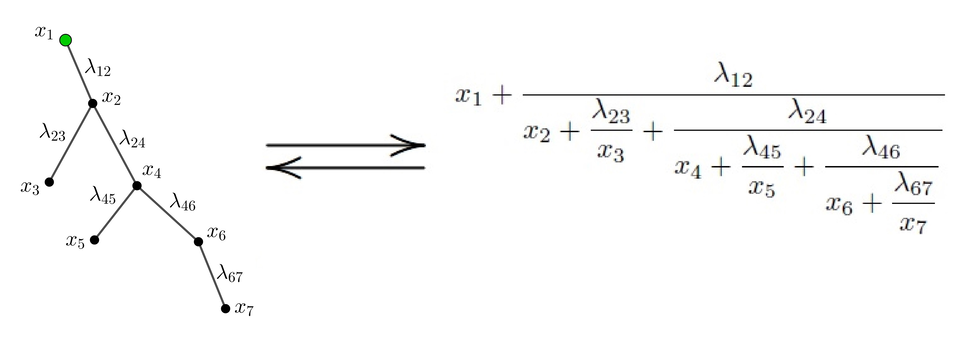}
	\caption{A rooted tree and its associated tree continued fraction.}
	\label{branchedcf}
\end{figure}

In this case, the following result holds. 

\begin{thm}\label{treecontinuedfraction} For a weighted tree $T$ with root $i$, its associated tree continued fraction $\alpha_i(T)$ is equal to $\dfrac{\mu(T)}{\mu(T\setminus i)}$.
\end{thm}
\begin{proof} For every graph $G$ and vertex $i$, by item $(b)$ of Lemma~\ref{recurrencesMatching}, it follows that,
	
\[\mu(G)=\ds\sum_{j\neq i}\lambda_{ij}\mu(G\setminus \{i,j\})+x_i\mu(G\setminus i)\iff \dfrac{\mu(G)}{\mu(G\setminus i)}=x_i+\ds\sum_{j\neq i}\dfrac{\lambda_{ij}}{\dfrac{\mu(G\setminus i)}{\mu(G\setminus \{i,j\})}}.\]
	
To finish the proof, we replace the graph $G$ with the tree $T$ in this last equation and observe that we get the same recurrence satisfied by $\alpha_i(T)$.
\end{proof}

Looking at the proof of Theorem~\ref{treecontinuedfraction} it can be seen that, in principle, it should work more generally for every rooted graph, the only missing ingredient being the analogue of a tree continued fraction. Iterating the recurrence for $\dfrac{\mu(G)}{\mu(G\setminus i)}$ for a rooted graph $G$, what one obtains at the end is a tree continued fraction for the {\it rooted path tree} of the rooted graph $G$.

For a rooted graph $G$ with root $i$ its {\it rooted path tree} $T^i_G$ is the rooted tree with vertices labeled by paths in $G$ starting at $i$, where two vertices are connected if one path is a maximal sub-path of the other. The root of $T^i_G$ is the trivial path $i$, and the weights of $T^i_G$ are obtained from the weights of $G$ as follows. If $c:i\to j$ is a maximal sub-path of $\hat{c}:i\to k$, then the weight of the edge connecting the vertices corresponding to the paths $c$ and $\hat{c}$ is $\lambda_{c\hat{c}} := \lambda_{jk}$. An example of a rooted graph with its corresponding rooted path tree is presented in Figure~\ref{godsil}.

This discussion motivates the following definition. 

\begin{dfn}(Graph continued fraction) The graph continued fraction of a weighted graph $G$ with root $i$ is defined as $\alpha_i(G):=\dfrac{\mu(G)}{\mu(G\setminus i)}$. 
\end{dfn}

Note that this is consistent with the definition of tree continued fraction. The observation above leads to the following lemma, originally due to Godsil~\cite[p. 287, Thm. 2.5]{godsil_matchings}. 

\begin{lem}(Godsil~\cite{godsil_matchings})\label{godsiltree} Given a rooted graph $G$ with root $i$ it holds:
	
\[\dfrac{\mu(G)}{\mu(G\setminus i)}=\alpha_i(G)=\alpha_i(T^i_G)=\dfrac{\mu(T^i_G)}{\mu(T^i_G\setminus i)}.\]
\end{lem}

As a consequence of this lemma, every graph continued fraction can be transformed into a tree continued fraction. This also allows the definition of graph continued fractions for infinite graphs. An illustration of Lemma~\ref{godsiltree} is presented in Figure~\ref{godsil}, where, for simplicity, the rooted graphs represent their graph continued fractions.

\begin{figure}[H]
	\includegraphics[width=\linewidth]{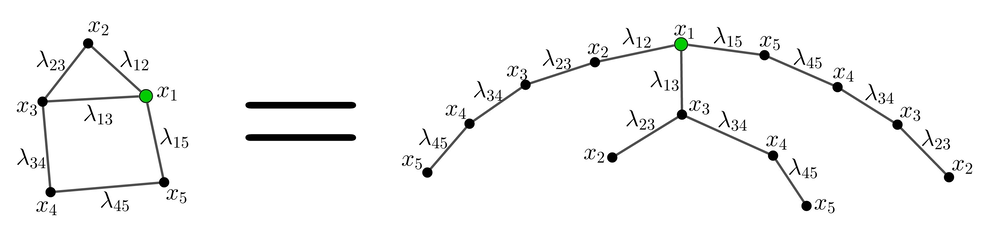}
	\caption{An illustration of the equality $\alpha_i(G)=\alpha_i(T^i_G)$.}
	\label{godsil}
\end{figure}

Using Lemma~\ref{godsiltree} we can give a proof of the classic result of Heilmann-Lieb~\cite[p. 201-203, Thms. 4.3 and 4.6]{heilmann-lieb} about the position of the zeros of multivariate matching polynomials.

\begin{thm}(Heilmann-Lieb~\cite{heilmann-lieb})\label{Heilmann-Lieb} The multivariate matching polynomial of $G$ is different from zero if one of the following conditions is satisfied: 
	\begin{itemize}
		\item $Im(x_i)>0$ for every $i$;
		\item $|x_i|>2\sqrt{B_G}$ for every $i$, where $B_G$ is equal to $\ds\max_j\,\max_{\mathclap{\substack{A\subseteq [n]\setminus j\\ |A|=n-2}}}\quad\sum_{k\in A}-\lambda_{jk}$ if $n\geq 3$, and equal to $-\lambda_{12}/4$ or $0$ if $n$ is two or one, respectively.
	\end{itemize}	
\end{thm}

\begin{proof} The approach is the same as in~\cite[p. 201-203, Thms. 4.3 and 4.6]{heilmann-lieb}. Consider a graph $G$ and let $R$ be one of the regions $[Im(x)>0]$ or $[|x|>2\sqrt{B_G}]$ in the complex plane. Our aim is to prove that $\mu(G)$ is different from zero in $R^n$. Note that for a graph with only one vertex this result is trivial. Assume, by induction hypothesis, that the statement is true for any graph with less vertices than $G$.
	
Choose any vertex $i$ as a root of $G$ and consider the graph continued fraction $\alpha_i(G)$. By the induction hypothesis, and $B_G\geq B_{G\setminus i}$, it is sufficient to prove that the graph continued fraction $\alpha_i(G)=\dfrac{\mu(G)}{\mu(G\setminus i)}$ is different from zero in $R^n$.
	
By Lemma~\ref{godsiltree}, $\alpha_i(G)$ is equal to the tree continued fraction $\alpha_i(T^i_G)$. Following the structure of the rooted tree $T^i_G$, one can write $\alpha_i(G)=\alpha_i(T^i_G)$ as a composition of some functions
	
\[f_{j,A}(x_1,\dots,x_n):=x_j+\ds\sum_{k\in A}\frac{\lambda_{jk}}{x_k},\]
	
\noindent with $j$ in $[n]$ and $A$ a subset of $[n]\setminus j$. Each function corresponding to a vertex in the rooted tree $T^i_G$. Observe that except for the last function in this composition, which corresponds to the root of $T^i_G$, all the other functions $f_{j,A}$ satisfy $|A|\leq n-2$. This can be seen by carefully examining the examples of Figures~\ref{branchedcf} and~\ref{godsil}.
	
Finally, observe that the image of $R^n$ by every function $f_{j,A}$ with $|A|\leq n-2$ is again contained in $R$, and that every function $f_{j,A}$ with $|A|=n-1$ is different from zero in $R^n$. Putting this all together it follows that $\alpha_i(G)=\alpha_i(T^i_G)$ is different from zero in $R^n$, which finishes the proof.
\end{proof}

With the concept of graph continued fraction already established, a natural follow up question is the effect of graph operations on a graph continued fraction. We consider one of the simplest graph operations there is, that of removing a vertex from the graph. Observe that,

\[\alpha_i(G)-\alpha_i(G\setminus j)=\dfrac{\mu(G)}{\mu(G\setminus i)}-\dfrac{\mu(G\setminus j)}{\mu(G\setminus \{j,i\})}=\]
\[=\dfrac{\mu(G\setminus \{i,j\})\mu(G)-\mu(G\setminus i)\mu(G\setminus j)}{\mu(G\setminus \{i,j\})\mu(G\setminus i)}.\]

Thus, we are led to consider the expression $\mu(G\setminus i)\mu(G\setminus j)-\mu(G\setminus \{i,j\})\mu(G)$. The next lemma, which originally appears in the work~\cite[p. 213, Thm. 6.3]{heilmann-lieb}, simplifies this last expression and is one of the main tools in the study of matching polynomials.

\begin{lem}\label{Christoffel-Darboux}(Christoffel-Darboux~\cite{heilmann-lieb}) Given a graph $G$ and two distinct vertices $i$ and $j$, it holds that
	
\[\mu(G\setminus i)\mu(G\setminus j)-\mu(G\setminus \{i,j\})\mu(G)=\ds\sum_{c\in[i\to j]}\lambda_c\cdot \mu(G\setminus c)^2,\]
	
\noindent where $[i\to  j]$ is the set of paths from $i$ to $j$ and $\lambda_c$ is the product of $-\lambda_e$ over the edges $e$ of the path $c$.
\end{lem}
\begin{proof} Proof by induction on the number of vertices of the graph, as in the work of Heilmann and Lieb~\cite[p. 213, Thm. 6.3]{heilmann-lieb}.
\end{proof} 

Observe that, since the edge weights are non-positive, $-\lambda_c$ is non-positive for every path $c$. Using Lemma~\ref{Christoffel-Darboux} we obtain a simplified formula for the effect of removing a vertex in a graph continued fraction: 

\[\alpha_i(G)-\alpha_i(G\setminus j)=\ds\dfrac{-\ds\sum_{c\in[i\to j]}\lambda_c\cdot \mu(G\setminus c)^2}{\mu(G\setminus \{i,j\})\mu(G\setminus i)}.\]

This generalizes the difference formula for the classical continued fractions. It turns out that it is useful to rewrite this last equality as follows.

\begin{lem}\label{Contraction}(Contraction) Given a graph $G$ and two distinct vertices $i$ and $j$, it holds that
	
\[\alpha_i(G)=\alpha_i(G\setminus j)+\dfrac{\lambda_{i\sim j}}{\alpha_j(G\setminus i)}, \text{ where } \lambda_{i\sim j}=\lambda_{j\sim i}:=\dfrac{-\ds\sum_{c\in[i\to j]}\lambda_c\cdot \mu(G\setminus c)^2}{\mu(G\setminus \{i,j\})^2}.\]
\end{lem}

Since all the edges have non-positive weights, $-\lambda_{i\sim j}$ is a sum of squares. Also, $\lambda_{i\sim j}$ does not depend on the vertex weights of $i$ and $j$.

The Lemma~\ref{Contraction} is well known in the classical theory of continued fractions, as can be seen in the books by Perron~\cite[p. 12, Satz 1.6]{perron1957lehreII} and Jones and Thron \cite[p. 38]{jones1984continued}. Using the theory of Heaps of Pieces~\cite[p. 149-150]{viennot1985combinatorial} it is possible to give a combinatorial interpretation to Lemma~\ref{Contraction}.

As will be seen in the next section, the Lemma~\ref{Contraction} is one of our main tools. It allows us to partially contract parts of a given graph continued fraction, thus simplifying expressions. This is particularly useful when comparing two graph continued fractions for the same graph but with different roots.


\section{Gallai-Edmonds Structure Theorem}\label{gallaiedmonds} 

In this section consider graphs with vertex set $[n]$ where the vertex weights are $x-r_i$, with $r_i$ a real number, and the edge weights are non-positive $\lambda_{ij}$. In this case the multivariate matching polynomial is the weighted matching polynomial as defined in Section~\ref{intro}. The weighted matching polynomial is a real univariate polynomial of degree $n$. The Theorem~\ref{Heilmann-Lieb} has the following consequence in this case.

\begin{cor}(Heilmann-Lieb~\cite{heilmann-lieb})\label{Heilmann-Lieb-real} All the zeros of the matching polynomial of $G$ are real and contained in the interval $\left[\ds\min_jr_j-2\sqrt{B_G},\ds\max_jr_j+2\sqrt{B_G}\right]$, where $B_G$ is equal to $\ds\max_j\,\max_{\mathclap{\substack{A\subseteq [n]\setminus j\\ |A|=n-2}}}\quad\sum_{k\in A}-\lambda_{jk}$ if $n\geq 3$, and equal to $-\lambda_{12}/4$ or $0$ if $n$ is two or one, respectively.	
\end{cor}
\begin{proof} As $\mu(G)$ has real coefficients, if it has a non-real zero, then by conjugation it has a zero in the upper half-plane. This is prohibited by Theorem~\ref{Heilmann-Lieb}, so $\mu(G)$ has only real zeros. The bound on the zeros follows immediately from the second item of Theorem~\ref{Heilmann-Lieb}.
\end{proof}

The particular case of equal edge weights in Corollary~\ref{Heilmann-Lieb-real} was used by Marcus, Spielman and Srivastava~\cite[p. 316, Thm. 5.5]{interlacing_familiesI} in their construction of bipartite Ramanujan graphs of all degrees.

The Corollary~\ref{Heilmann-Lieb-real} implies that $\alpha_i(G)(x)$ is a real rational function with all its zeros and poles in the real line. In order to better understand the position of the zeros and poles we look at the derivative of matching polynomials and graph continued fractions.

\begin{lem}\label{derivative0} Let $G$ be a rooted graph with root $i$. Then:
\begin{itemize}
	\item $\mu(G)'(x)=\ds\sum_{j\in [n]}\mu(G\setminus j)(x)$;
	\item $\alpha_i(G)'(x)=1+\ds\sum_{i\neq j\in [n]}\ds\sum_{c\in[i\to j]}\lambda_c\cdot \left(\dfrac{\mu(G\setminus c)}{\mu(G\setminus i)}(x)\right)^2=1-\ds\sum_{i\neq j\in [n]}\dfrac{\lambda_{i\sim j}(x)}{(\alpha_j(G\setminus i)(x))^2}$.
\end{itemize}
\end{lem}
\begin{proof} The first item is a direct consequence of item $(d)$ of Lemma~\ref{recurrencesMatching}. For an alternative proof see Ku and Wong's work~\cite[p. 3390, Thm. 2.1]{ku2013gallai}.

For the second item consider the derivative of the recurrence:

\[\alpha_i(G)(x)=x-r_i+\ds\sum_{i\neq j}\dfrac{\lambda_{ij}}{\alpha_j(G\setminus i)(x)}\implies \alpha_i(G)'(x)=1+\ds\sum_{i\neq j}\dfrac{-\lambda_{ij}\alpha_j(G\setminus i)'(x)}{(\alpha_j(G\setminus i)(x))^2}.\]

Iterating the recurrence for the derivative the second item immediately follows. An alternative proof can be given using the first item and the Lemma~\ref{Christoffel-Darboux}.
\end{proof}

\begin{cor}\label{derivative1} Let $G$ be a rooted graph with root $i$. Then all the zeros and poles of $\alpha_i(G)$ are simple. If $\theta$ is not a pole of $\alpha_i(G)$, then $\alpha_i(G)'(\theta)\geq 1$. In particular, $\alpha_i(G)(x)$ is increasing and surjective in each of its branches.
\end{cor}
\begin{proof} If $\mu(G\setminus i)(\theta)\neq 0$, then the Lemma~\ref{derivative0} implies that $\alpha_i(G)'(\theta)\geq 1$.
Note that $\alpha_i(G)'$ is continuous at every $\theta$ that is not a pole of $\alpha_i(G)$. It follows then that $\alpha_i(G)'(\theta)\geq 1$ for every $\theta$ that is not a pole of $\alpha_i(G)$. In particular, $\alpha_i(G)$ is increasing and surjective in each of its branches and all of its zeros are simple. 

Observe that, since $\deg(\mu(G))=\deg(\mu(G\setminus i))+1$, the number of zeros of $\alpha_i(G)$ is one more than the number of poles counted with multiplicity of $\alpha_i(G)$. But in each branch, because $\alpha_i(G)$ is increasing, there can only be one zero of $\alpha_i(G)$. Putting this all together, it follows that all the poles of $\alpha_i(G)$ are also simple.
\end{proof}

The last result gives a precise picture of how a graph of $\alpha_i(G)(x)$ must look like. In Figure~\ref{graph} we present an example of such a graph.

\begin{figure}[h]
	\includegraphics[width=\linewidth]{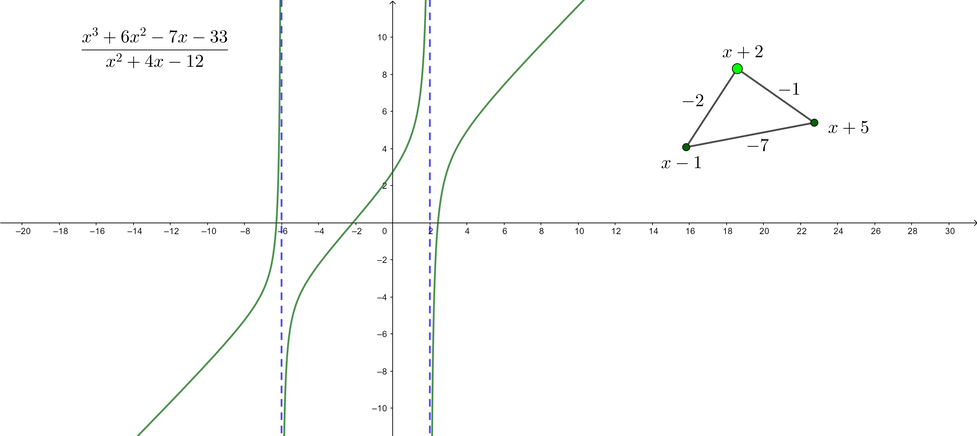}
	\caption{An example of a graph of $\alpha_i(G)(x)$.}
	\label{graph}
\end{figure}

Corollary~\ref{derivative1} also implies the interlacing for the zeros of $\mu(G)$ and $\mu(G\setminus i)$ mentioned in Section~\ref{intro}. This result was originally proved by Heilmann and Lieb~\cite[p. 200, Thm. 4.2]{heilmann-lieb}.

\begin{cor}(Interlacing~\cite{heilmann-lieb})\label{interlacing} Let $i$ be a vertex in the graph $G$. Then the zeros of $\mu(G)$ and $\mu(G\setminus i)$ interlace, i.e. between any two zeros of $\mu(G)$ there is a zero of $\mu(G\setminus i)$ and vice versa. It is also true that $m_\theta(G\setminus i)$ belongs to $\{m_\theta(G),m_\theta(G)\pm 1\}$ for every real number $\theta$.
\end{cor}
\begin{proof} By Corollary~\ref{derivative1} the zeros and poles of $\alpha_i(G)$ are simple. This implies that $m_\theta(G\setminus i)$ belongs to $\{m_\theta(G),m_\theta(G)\pm 1\}$ for every real number $\theta$. The interlacing of the zeros of $\mu(G)$ and $\mu(G\setminus i)$ follows from the interlacing of the zeros and poles of $\alpha_i(G)$ and this last observation about the multiplicities $m_\theta(G)$ and $m_\theta(G\setminus i)$.
\end{proof}

Given a real parameter $\theta$, partition the vertices of the graph $G$ into four sets according to the sign of the graph continued fraction with each vertex as a root. That is, if $i$ is a vertex then:

\begin{itemize}
	\item $i\in -_{\theta,G}$ if $\alpha_i(G)(\theta)$ is negative;
	\item $i\in 0_{\theta,G}$ if $\alpha_i(G)(\theta)$ is zero;
	\item $i\in +_{\theta,G}$ if $\alpha_i(G)(\theta)$ is positive;
	\item $i\in \infty_{\theta,G}$ if $\alpha_i(G)(\theta)$ is infinite.
\end{itemize}

This way we have the partition $[n]=-_{\theta,G}\sqcup 0_{\theta,G}\sqcup +_{\theta,G}\sqcup\infty_{\theta,G}$. Define also $\pm_{\theta,G}:=-_{\theta,G}\sqcup +_{\theta,G}$. Note that by Corollary~\ref{interlacing}:

\begin{itemize}
	\item $i\in 0_{\theta,G}$ if $m_\theta(G\setminus i)=m_\theta(G)-1$;
	\item $i\in \pm_{\theta,G}$ if $m_\theta(G\setminus i)=m_\theta(G)$;
	\item $i\in \infty_{\theta,G}$ if $m_\theta(G\setminus i)=m_\theta(G)+1$.
\end{itemize}

It follows that with the notation of Section~\ref{intro}: $0_{\theta,G}=D_{\theta,G}$, $\pm_{\theta,G}=N_{\theta,G}$, $A_{\theta,G}=\partial 0_{\theta,G}$ and $\infty_{\theta,G}\setminus \partial 0_{\theta,G}=P_{\theta,G}$. This shows that the partition $[n]=-_{\theta,G}\sqcup 0_{\theta,G}\sqcup +_{\theta,G}\sqcup\infty_{\theta,G}$ refines the one considered by Ku and Wong in~\cite[p. 3389]{ku2013gallai}, where there was no distinction between $+_{\theta,G}$ and $-_{\theta,G}$.

From Corollary~\ref{derivative1}, or looking at Figure~\ref{graph}, it can be seen that as the parameter $\theta$ increases from $-\infty$ to $+\infty$ the sign of $\alpha_i(G)(\theta)$ always changes in a prescribed order: $-\rightarrow 0\rightarrow +\rightarrow\infty\rightarrow -$. This already shows that as $\theta$ is varied the partitions of $[n]$ change according to some rules. The parameter $\theta$ is seen as a time variable determining the values of the graph continued fractions and partitions of $[n]$.

Clearly, if $\theta$ is not a zero of $\mu(G)$ then the set $0_{\theta,G}$ is empty. As observed by
Godsil~\cite[p. 5, Lem. 3.1]{godsil1995algebraic}, it turns out that the converse is also true.

\begin{lem}(Godsil~\cite{godsil1995algebraic})\label{zeroset} The real number $\theta$ is a zero of $\mu(G)$ if, and only if, $0_{\theta,G}$ is non-empty.
\end{lem}
\begin{proof} If $\theta$ is a zero of $\mu(G)$, then $\infty=\dfrac{\mu(G)'}{\mu(G)}(\theta)=\ds\sum_{j\in[n]}\dfrac{\mu(G\setminus j)}{\mu(G)}(\theta)=\ds\sum_{j\in[n]}\dfrac{1}{\alpha_j(G)(\theta)}$, which implies that there exists a vertex $j$ satisfying $\alpha_j(G)(\theta)=0$, i.e. $j\in 0_{\theta,G}$.
\end{proof}

The same proof of this last lemma implies:

\begin{lem}\label{internal} A vertex $i$ is in $\infty_{\theta,G}$ if, and only if, one of its neighbors is in $0_{\theta,G\setminus i}$.
\end{lem}
\begin{proof} Observe that, $\infty=\alpha_i(G)(\theta)=\theta-r_i+\ds\sum_{i\neq j}\dfrac{\lambda_{ij}}{\alpha_j(G\setminus i)(\theta)}$ if, and only if, there exists a vertex $j$ satisfying $\lambda_{ij}\neq 0$ and $\alpha_j(G\setminus i)(\theta)=0$, i.e. $j$ is a neighbor of $i$ that belongs to $0_{\theta,G\setminus i}$.
\end{proof}

This last result is best interpreted using the path tree. From Corollary~\ref{derivative1}, or looking at Figure~\ref{graph}, it is clear that a vertex is in $0_{\theta,G}$ if, and only if, it is in the intersection $-_{\theta-\epsilon,G}\cap +_{\theta+\epsilon,G}$ for every $\epsilon>0$ sufficiently small. This means that a vertex $i$ is in $0_{\theta,G}$ if, and only if, $\alpha_i(G)(x)$ changes sign from $-$ to $+$ at time $\theta$. A similar reasoning applies for $\infty_{\theta,G}$.

Using the recurrence $\alpha_i(G)(x)=x-r_i+\ds\sum_{\mathclap{\substack{i\neq j\\ \lambda_{ij}\neq 0}}}\quad\dfrac{\lambda_{ij}}{\alpha_j(G\setminus i)(x)}$, the Lemma~\ref{internal} can be interpreted as saying that $\alpha_i(G)(x)$ changes sign from  $+$ to $-$ at time $\theta$ if, and only if, for some neighbor $j$ of $i$, $\alpha_j(G\setminus i)(x)$ changes sign from $-$ to $+$ at time $\theta$. 

Consider the path tree $T^i_G$. Recall that each path starting at $i$ corresponds to a vertex in the path tree $T^i_G$. For a given path $c$ starting at $i$ if we consider all the paths that are extensions of $c$ we obtain a rooted subtree $T^c_G$ of $T^i_G$ with root $c$. If the path is $c:i=i_1\to i_k$ then it is easy to see that $\alpha_c(T^c_G)$ is equal to $\alpha_{i_k}(G\setminus \{i_1,\dots,i_{k-1}\})$. 

Write for each vertex of the path tree $T^i_G$ the sign of the graph continued fraction for its respective rooted subtree at time $\theta$, i.e. for the vertex corresponding to the path $c$ consider the sign of the graph continued fraction $\alpha_c(T^c_G)(\theta)$. This is illustrated in Figure~\ref{pinheiro}.

\begin{figure}[h]
	\includegraphics[width=\linewidth]{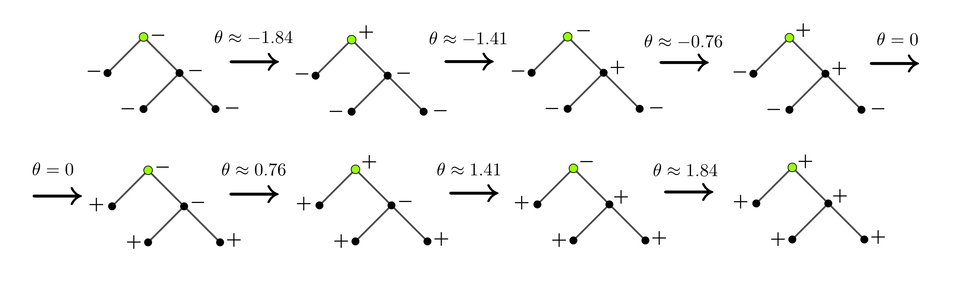}
	\caption{A rooted tree with vertex weights $x$ and edge weights $-1$. The signs are of the graph continued fractions of the subtrees. As time passes the plus signs fill in the tree.}
	\label{pinheiro}
\end{figure}

Observe that, by Corollary~\ref{derivative1}, for large negative times all subtrees have sign $-$, and for large positive times all subtrees have sign $+$. As the time $\theta$ goes by the $+$ signs are created at the root of the path tree and descend, sometimes duplicating, but always respecting the following rule: a node changes from $+$ to $-$ at time $\theta$ if, and only if, one of its subnodes changes from $-$ to $+$ at time $\theta$. This is the promised interpretation of Lemma~\ref{internal} in terms of the path tree. An illustration of how these signals evolve with the time $\theta$ is also shown in Figure~\ref{pinheiro}.

This interpretation in terms of the path tree is particularly interesting when studying paths. Let $c:i=i_1\to i_k$ be a path in the graph $G$. Observe that,

\[\dfrac{\mu(G)}{\mu(G\setminus c)}=\alpha_{i_1}(G)\alpha_{i_2}(G\setminus i_1)\cdots \alpha_{i_k}(G\setminus \{i_1,\dots,i_{k-1}\})=\]

\[=\alpha_{i_k}(G)\alpha_{i_{k-1}}(G\setminus i_k)\cdots \alpha_{i_1}(G\setminus \{i_k,i_{k-1},\dots,i_2\}).\]

It turns out that the difference of multiplicities $m_\theta(G)-m_\theta(G\setminus c)$ can also be interpreted in terms of the signs in the path tree at time $\theta$. The multiplicity $m_\theta(G)-m_\theta(G\setminus c)$ is equal to the number of zeros minus the number of infinities along the sub-paths of the path $c$ in the path tree $T^i_G$ at time $\theta$. More precisely,

\[m_\theta(G)-m_\theta(G\setminus c)=|\{i_j\in 0_{\theta,G\setminus \{i_1,\dots,i_{j-1}\}}\}_{j\in [k]}|-|\{i_j\in \infty_{\theta,G\setminus \{i_1,\dots,i_{j-1}\}}\}_{j\in [k]}|=\]

\[=|\{i_j\in 0_{\theta,G\setminus \{i_k,\dots,i_{j+1}\}}\}_{j\in [k]}|-|\{i_j\in \infty_{\theta,G\setminus \{i_k,\dots,i_{j+1}\}}\}_{j\in [k]}|.\]

The second equality corresponds to the same statement but for the reverse path $-c:i_k\to i_1$. In particular, this shows that the difference of zeros and infinities along the sub-paths of the path $c:i=i_1\to i_k$ in the path tree $T^i_G$ is equal to the same counting for the path $-c:i_k\to i_1=i$ in the path tree $T^{i_k}_G$.

As a consequence we have the following result, originally proved by Godsil~\cite[p. 4-6, Cor. 2.5 and Lem. 3.3]{godsil1995algebraic}.

\begin{lem}[Godsil~\cite{godsil1995algebraic}]\label{entrylessthan1} Let $c:i\to j$ be a path in the graph $G$. In this case, $m_\theta(G)-m_\theta(G\setminus c)\leq 1$, and, if there is equality, then both $i$ and $j$ are in $0_{\theta,G}$.
\end{lem}
\begin{proof} Let $c:i_1\to i_m$ be a path in the graph $G$. First, notice that, by Lemma~\ref{internal} or its interpretation in terms of the path tree, whenever there is a zero in a node of the path tree $T^{i_1}_G$ there must be an infinity for the node right above it. In other words, if $\alpha_{i_{k+1}}(G\setminus \{i_1,\dots,i_k\})=0$ for some $k\in [m-1]$, then $\alpha_{i_k}(G\setminus \{i_1,\dots,i_{k-1}\})=\infty$. This implies that the number of zeros is less than the number of infinities along the sub-paths of the path $c$ in the path tree $T^{i_1}_G$, from which follows that $m_\theta(G)-m_\theta(G\setminus c)\leq 1$. 

If the path $c$ satisfies $m_\theta(G)-m_\theta(G\setminus c)=1$, then, by the same reasoning above, all the infinities along the sub-paths of the path $c$ in $T_G^{i_1}$ come from a zero inside the same path $c$. But there is also one extra zero which does not have a corresponding infinity. Since the extra zero does not have a corresponding infinity, by the observation above, it must be at the root of the path tree $T_G^{i_1}$. This implies that $i_1$ is in $0_{\theta,G}$. The same reasoning for the reverse path $-c$ implies that $i_m$ is also in $0_{\theta,G}$. 
\end{proof}

Observe that except for a finite number of times $\theta$, namely those that are a zero of a matching polynomial of a subgraph of the graph $G$, there are only plus and minus signs for all subtrees of the path tree $T^i_G$. Consider a time $\theta$ with this property. In this case, there is also an interpretation for the number of plus signs along the sub-paths of the path $c:i=i_1\to i_k$ in the path tree $T^i_G$.

The variation of the number of plus signs along the sub-paths of the path $c$ in $T^i_G$ at time $\theta$ is equal to $m_\theta(G)-m_\theta(G\setminus c)$. This shows that generically:

\[\ds\sum_{x<\theta}(m_x(G)-m_x(G\setminus c))=|\{i_j\in+_{\theta,G\setminus \{i_1,\dots,i_{j-1}\}}\}_{j\in [k]}|=|\{i_j\in+_{\theta,G\setminus \{i_k,\dots,i_{j+1}\}}\}_{j\in [k]}|.\]

The fact that $m_\theta(G)-m_\theta(G\setminus c)\leq 1$, proved in Lemma~\ref{entrylessthan1}, means that the number of plus signs along the sub-paths of the path $c$ in $T^i_G$ can increase by at most one at time $\theta$.
This is clear because the plus signs can only enter the path $c$ through the root of $T^i_G$. If $m_\theta(G)-m_\theta(G\setminus c)=1$, then the number of plus signs increases by one at time $\theta$.

This interpretation leads to the following result, of which the last part appears already in the work of Godsil~\cite[p. 296, Cor. 5.3]{godsil_matchings}.

\begin{lem}\label{simultaneous} Let $c:i\to j$ be a path of length $k$ in the graph $G$. Then, there are at least $k+1$ distinct times $\theta$ such that both $i$ and $j$ are simultaneously in $0_{\theta,G}$. In particular, $\mu(G)$ has at least $k+1$ distinct zeros.
\end{lem}
\begin{proof} Note that to fill in the path $c:i=i_1\to i_{k+1}=j$ with plus signs in $T^i_G$, the path $c$ must satisfy $m_\theta(G)-m_\theta(G\setminus c)= 1$ for at least $k+1$ distinct times $\theta$. By Lemma~\ref{entrylessthan1} at these times $\theta$ both $i$ and $j$ are simultaneously in $0_{\theta,G}$.
\end{proof}

For hamiltonian paths this result has the following immediate corollary.

\begin{cor}[Godsil~\cite{godsil_matchings}]\label{hamiltonian} Let $c$ be a hamiltonian path in the graph $G$. Then, $\mu(G)$ has distinct zeros.
\end{cor}

Now, consider a hamiltonian path $c:i=i_1\to i_n$ in the graph $G$. Denote by $V_c(\theta)$ the number of plus signs along the sub-paths of the hamiltonian path $c$ in the path tree $T^i_G$ at time $\theta$, i.e. $V_c(\theta)=|\{i_j\in+_{\theta,G\setminus \{i_1,\dots,i_{j-1}\}}\}_{j\in [n]}|$. Similarly, denote by $V_{-c}(\theta)$ the analogous counting for the reverse path $-c$. Since $G\setminus c=\emptyset$, we obtain,  

\[\ds\sum_{x<\theta}m_x(G)=V_c(\theta)=V_{-c}(\theta),\].

But by Corollary~\ref{hamiltonian} the zeros of $\mu(G)$ are distinct. Putting all this together, we obtain the Sylvester Theorem~\ref{newsylvester} mentioned in Section~\ref{intro}. 

\begin{thm}(Sylvester theorem) Let $c$ be a Hamiltonian
path in the graph $G$. Then, $\mu(G)$ has distinct zeros and both $V_c(\theta)$ and $V_{-c}(\theta)$, when defined, are equal to the number of zeros of $\mu(G)$ in the interval $(-\infty,\theta)$.
\end{thm}

We now focus on the the Stability Lemma~\ref{newstability}. In order to prove the Stability Lemma~\ref{newstability} we must study how a graph continued fraction changes when a vertex is deleted. To approach this problem we use the contraction Lemma~\ref{Contraction}. 

For distinct vertices $i$ and $j$ it was previously observed that $-\lambda_{i\sim j}$ is a sum of squares. So $\lambda_{i\sim j}(\theta)$ is in $[-\infty,0]$, but it can happen that $\lambda_{i\sim j}$ is equal to $-\infty$ at time $\theta$. The next lemma gives a condition that guarantees $\lambda_{i\sim j}(\theta)$ is finite.

\begin{pps}\label{finitelambda} If $\lambda_{i\sim j}(\theta)=-\infty$, then $i\in\infty_{\theta,G\setminus j}$ and $j\in\infty_{\theta,G\setminus i}$. 
\end{pps}
\begin{proof} Assume $i$ is not in $\infty_{\theta,G\setminus j}$. In order to verify that $\lambda_{i\sim j}$ is finite it is sufficient to prove that $m_\theta(G\setminus \{i,j\})\leq m_\theta(G\setminus c)$ for every path $c$ in $[i\rightarrow j]$. Consider a path $c:i=i_1\to i_k=j$ between $i$ and $j$, and let $\overline{c}:i=i_1\to i_{k-1}$ be the path in $G\setminus j$ obtained from $c$.
	
If $i$ is in $0_{\theta,G\setminus j}$, then, by Lemma~\ref{entrylessthan1}, it holds, $m_\theta(G\setminus j)-m_\theta(G\setminus (\{j\}\sqcup\overline{c}))\leq 1\implies$

\[m_\theta(G\setminus \{j,i\})=m_\theta(G\setminus j)-1\leq m_\theta(G\setminus (\{j\}\sqcup\overline{c}))=m_\theta(G\setminus c)\implies\]

\[m_\theta(G\setminus \{i,j\})\leq m_\theta(G\setminus c).\]
	
Assume now that the vertex $i$ is in $\pm_{\theta,G\setminus j}$. In this case, by Lemma~\ref{entrylessthan1}, the path $\overline{c}:i=i_1\to i_{k-1}$ satisfies $m_\theta(G\setminus j)-m_\theta(G\setminus (\{j\}\sqcup\overline{c}))\leq 0$. This implies that,

\[m_\theta(G\setminus \{j,i\})-m_\theta(G\setminus c)=m_\theta(G\setminus j)-m_\theta(G\setminus (\{j\}\sqcup\overline{c}))\leq 0\implies\]

\[m_\theta(G\setminus \{i,j\})\leq m_\theta(G\setminus c).\] 
\end{proof} 

If $\lambda_{i\sim j}(\theta)$ is equal to zero, then we have the following useful observation.

\begin{pps}\label{relative0} If $\lambda_{i\sim j}(\theta)=0$, then $\alpha_i(G)(\theta)=\alpha_i(G\setminus j)(\theta)$ and $\alpha_j(G)(\theta)=\alpha_j(G\setminus i)(\theta)$.
\end{pps}
\begin{proof} The Contraction Lemma~\ref{Contraction} implies that $\alpha_i(G)(\theta)=\alpha_i(G\setminus j)(\theta)+\dfrac{\lambda_{i\sim j}(\theta)}{\alpha_j(G\setminus i)(\theta)}$ and $\alpha_j(G)(\theta)=\alpha_j(G\setminus i)(\theta)+\dfrac{\lambda_{j\sim i}(\theta)}{\alpha_i(G\setminus j)(\theta)}$. Observe that, since $-\lambda_{i\sim j}$ is a sum of squares, $\lambda_{i\sim j}$ has double zeros. On the other hand, by Proposition~\ref{derivative1}, $\alpha_i(G\setminus j)$ and $\alpha_j(G\setminus i)$ have simple zeros. It follows that, if $\lambda_{i\sim j}(\theta)=0$, then both $\dfrac{\lambda_{i\sim j}}{\alpha_j(G\setminus i)}(\theta)$ and $\dfrac{\lambda_{i\sim j}}{\alpha_i(G\setminus j)}(\theta)$ are equal to $0$, which in turn implies that $\alpha_i(G)(\theta)=\alpha_i(G\setminus j)(\theta)$ and $\alpha_j(G)(\theta)=\alpha_j(G\setminus i)(\theta)$.
\end{proof}

Now, if $\lambda_{i\sim j}(\theta)$ is in $(-\infty, 0)$, then we have more cases.

\begin{pps}\label{relative1} Consider $\lambda_{i\sim j}(\theta)\in (-\infty,0)$. In this case:
\begin{enumerate}[label=(\alph*)]	
	\item if $i\in +_{\theta,G\setminus j}$ and $j\in +_{\theta,G\setminus i}$, or $i\in -_{\theta,G\setminus j}$ and $j\in -_{\theta,G\setminus i}$, then $i$ and $j$ are simultaneously in either one of $-_{\theta,G}$, $0_{\theta,G}$ or $+_{\theta,G}$;
	\item if $i\in +_{\theta,G\setminus j}$ and $j\in -_{\theta,G\setminus i}$, then $i\in +_{\theta,G}$ and $j\in -_{\theta,G}$;
	\item if $i\in 0_{\theta,G\setminus j}$ and $j\in 0_{\theta,G\setminus i}$, then $i,j\in\infty_{\theta,G}$;
	\item if $i\in 0_{\theta,G\setminus j}$ and $j\in +_{\theta,G\setminus i}$, then $i\in -_{\theta,G}$ and $j\in \infty_{\theta,G}$;
	\item if $i\in 0_{\theta,G\setminus j}$ and $j\in -_{\theta,G\setminus i}$, then $i\in +_{\theta,G}$ and $j\in \infty_{\theta,G}$;
	\item if $i\in \infty_{\theta,G\setminus j}$, then $\alpha_j(G\setminus i)(\theta)=\alpha_j(G)(\theta)$ and $i\in \infty_{\theta,G}$.
\end{enumerate}
\end{pps}
\begin{proof} Consider $i\in +_{\theta,G\setminus j}$ and $j\in +_{\theta,G\setminus i}$, the other cases being analogous. The Contraction Lemma~\ref{Contraction} implies that $\alpha_i(G)(\theta)=\alpha_i(G\setminus j)(\theta)+\dfrac{\lambda_{i\sim j}(\theta)}{\alpha_j(G\setminus i)(\theta)}$ and $\alpha_j(G)(\theta)=\alpha_j(G\setminus i)(\theta)+\dfrac{\lambda_{j\sim i}(\theta)}{\alpha_i(G\setminus j)(\theta)}$. It then follows that $\alpha_i(G)(\theta)$ and $\alpha_j(G)(\theta)$ are both finite and have the same sign. This shows that $i$ and $j$ are simultaneously in either one of  $-_{\theta,G}$, $0_{\theta,G}$ or $+_{\theta,G}$. 
\end{proof}

The next result shows that we can indeed obtain inequalities in the items $(a)$ and $(b)$ of the Proposition~\ref{relative1}.

\begin{pps}\label{relative3} Let $\lambda_{i\sim j}(\theta)\in (-\infty,0)$. In this case:
\begin{enumerate}[label=(\alph*)]
	\item if $i\in +_{\theta,G\setminus j}\cap +_{\theta,G}$ and $j\in +_{\theta,G\setminus i}\cap +_{\theta,G}$, then $0<\alpha_i(G)(\theta)<\alpha_i(G\setminus j)(\theta)$ and $0<\alpha_j(G)(\theta)<\alpha_j(G\setminus i)(\theta)$;
	\item if $i\in -_{\theta,G\setminus j}\cap -_{\theta,G}$ and $j\in -_{\theta,G\setminus i}\cap -_{\theta,G}$, then $\alpha_i(G\setminus j)(\theta)<\alpha_i(G)(\theta)<0$ and $\alpha_j(G\setminus i)(\theta)<\alpha_j(G)(\theta)<0$;
	\item if $i\in +_{\theta,G\setminus j}\cap +_{\theta,G}$ and $j\in -_{\theta,G\setminus i}\cap -_{\theta,G}$, then $\alpha_i(G\setminus j)(\theta)<\alpha_i(G)(\theta)$ and $\alpha_j(G)(\theta)<\alpha_j(G\setminus i)(\theta)$.
\end{enumerate}
\end{pps}
\begin{proof} The proof uses the Contraction Lemma~\ref{Contraction} and is analogous to that of Proposition~\ref{relative1}.
\end{proof}

Using Proposition~\ref{relative1} we can also get further conditions that guarantee that $\lambda_{i\sim j}(\theta)$ is finite.

\begin{pps}\label{relative2} If $\lambda_{i\sim j}(\theta)=-\infty$, then $i$ and $j$ are simultaneously in either one of $-_{\theta,G}$, $0_{\theta,G}$, $+_{\theta,G}$ or $\infty_{\theta,G}$.
\end{pps}
\begin{proof} If $\lambda_{i\sim j}(\theta)=-\infty$, then Proposition~\ref{finitelambda} implies $i\in\infty_{\theta,G\setminus i}$ and $j\in\infty_{\theta,G\setminus i}$. It follows that $i\in +_{\theta-\epsilon,G\setminus j}\cap -_{\theta+\epsilon,G\setminus j}$ and $j\in +_{\theta-\epsilon,G\setminus i}\cap -_{\theta+\epsilon,G\setminus i}$ for every $\epsilon>0$ sufficiently small. Since $\lambda_{i\sim j}(\theta-\epsilon)\neq -\infty$, $i\in +_{\theta-\epsilon,G\setminus j}$ and $j\in +_{\theta-\epsilon,G\setminus i}$, the item $(a)$ of Proposition~\ref{relative1} implies that $\alpha_i(G)(\theta-\epsilon)$ and $\alpha_j(G)(\theta-\epsilon)$ have the same sign for every $\epsilon>0$ small. Similarly, $\alpha_i(G)(\theta+\epsilon)$ and $\alpha_j(G)(\theta+\epsilon)$ have the same sign for every $\epsilon>0$ small. As a consequence, $\alpha_i(G)(\theta)$ and $\alpha_j(G)(\theta)$ have the same sign, which finishes the proof.
\end{proof}

The content of Propositions~\ref{finitelambda},~\ref{relative0},~\ref{relative1} and~\ref{relative2} is summarized in the Figure~\ref{table}.

\begin{figure}[h]
	\includegraphics[width=\linewidth]{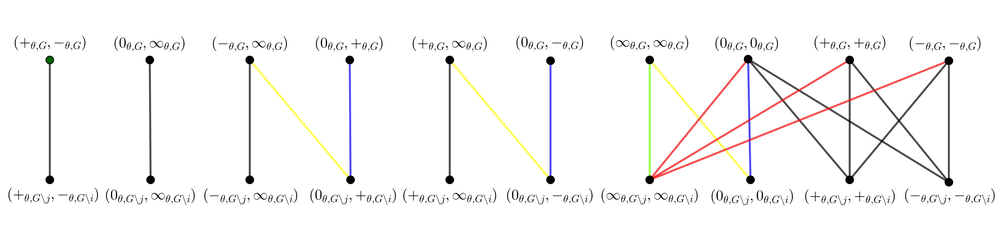}
	\caption{The nodes represent possible signs for the pair of distinct vertices $(i,j)$, both in $G$ and in $G\setminus j$ and $G\setminus i$. The edges join signs configurations that can occur simultaneously. The green, black and yellow edges represent $\lambda_{i\sim j}$ in $[-\infty,0]$, $(-\infty,0]$ and $(-\infty,0)$, respectively. The red and blue edges represent $\lambda_{i\sim j}$ equal to $-\infty$ and $0$, respectively.}
	\label{table}
\end{figure}

The next example illustrates how Figure~\ref{table} summarizes the content of the last propositions.

\begin{exm} Assume that, $i\in 0_{\theta, G}$ and $j\in +_{\theta, G}$. By Proposition~\ref{relative2}, we have that $\lambda_{i\sim j}(\theta)$ is finite. Now, note that, by Proposition~\ref{relative1}, $\lambda_{i\sim j}(\theta)$ cannot be in $(-\infty, 0)$. Indeed, Proposition~\ref{relative1} shows that there is no way to obtain $i\in 0_{\theta, G}$ and $j\in +_{\theta, G}$ if $\lambda_{i\sim j}(\theta)$ is in $(-\infty, 0)$. It follows that $\lambda_{i\sim j}(\theta)$ is equal to zero. But then, by Proposition~\ref{relative0}, we obtain that $i\in 0_{\theta, G\setminus j}$ and $j\in +_{\theta, G\setminus i}$.

This sequence of deductions is summarized in Figure~\ref{table} as follows. The information that $i\in 0_{\theta, G}$ and $j\in +_{\theta, G}$ is represented by the vertex with label $(0_{\theta, G}, +_{\theta, G})$ in Figure~\ref{table}. The fact that $i\in 0_{\theta, G\setminus j}$ and $j\in +_{\theta, G\setminus i}$ is represented by the vertex with label $(0_{\theta, G\setminus j}, +_{\theta, G\setminus i})$ in Figure~\ref{table}. Observe in Figure~\ref{table} that the only neighbor of the vertex with label $(0_{\theta, G}, +_{\theta, G})$ is the vertex with label $(0_{\theta, G\setminus j}, +_{\theta, G\setminus i})$. This means that if $i\in 0_{\theta, G}$ and $j\in +_{\theta, G}$, then the only possibility is that $i\in 0_{\theta, G\setminus j}$ and $j\in +_{\theta, G\setminus i}$. Furthermore, the fact that the edge connecting the vertices with labels $(0_{\theta, G}, +_{\theta, G})$ and $(0_{\theta, G\setminus j}, +_{\theta, G\setminus i})$ in Figure~\ref{table} is blue means that in this described situation $\lambda_{i\sim j}(\theta)$ must be equal to zero.
\end{exm}

The {\it frontier} of a subset of vertices $A$ of $[n]$, denoted by $\partial A$, is defined as the set of vertices that are not in $A$ but have a neighbor in $A$.

\begin{pps}\label{frontier} The frontier $\partial 0_{\theta,G}$ is a subset of $\infty_{\theta,G}$.
\end{pps}
\begin{proof} Let $i$ be in $\partial 0_{\theta,G}$ with a neighbor $j$ in $0_{\theta,G}$. Since $i$ and $j$ are neighbors, $\lambda_{i\sim j}(\theta)$ is non-zero. By Figure~\ref{table} this implies that $i$ cannot be in $\pm_{\theta,G}$, so it must be in $\infty_{\theta,G}$.
\end{proof}

The next result is simply a reformulation of Proposition~\ref{frontier}.

\begin{cor} Let $i$ and $j$ be neighbors in the graph $G$. If $\alpha_i(G)$ changes sign from $-$ to $+$ at time $\theta$, then $\alpha_j(G)$ changes sign from $-$ to $+$ or from $+$ to $-$ at time $\theta$. 
\end{cor}
\begin{proof} If $\alpha_i(G)$ changes sign from $-$ to $+$ at time $\theta$, then $i\in 0_{\theta,G}$. This implies by Proposition~\ref{frontier} that $j\in 0_{\theta,G}\sqcup \infty_{\theta,G}$, from which the result follows.
\end{proof}

Now we are ready to prove our main result, the Stability Lemma~\ref{newstability}.

\begin{pps}\label{preservespm} If $i\in \partial 0_{\theta,G}$ and $j\in\pm_{\theta,G}$ then $\alpha_j(G\setminus i)(\theta)=\alpha_j(G)(\theta)$. As a consequence, $-_{\theta,G\setminus i}=-_{\theta,G}$ and $+_{\theta,G\setminus i}=+_{\theta,G}$.
\end{pps}
\begin{proof}Let $k\in 0_{\theta,G}$ be a neighbor of $i\in\partial 0_{\theta, G}$. By the Figure~\ref{table}, since $j\in\pm_{\theta,G}$, it holds that $k\in 0_{\theta,G\setminus j}$ and $i\not\in 0_{\theta,G\setminus j}$. But the vertices $k$ and $i$ are also neighbors in $G\setminus j$, and from Proposition~\ref{frontier} we have $\partial 0_{\theta,G\setminus j}\subseteq \infty_{\theta,G\setminus j}$, so it must be $i\in\infty_{\theta,G\setminus j}$. Using item $(f)$ of Proposition~\ref{relative1} the result immediately follows.
\end{proof}

\begin{thm}(Stability lemma)\label{stability} If $i\in \partial 0_{\theta,G}$, then $\alpha_j(G\setminus i)(\theta)=\alpha_j(G)(\theta)$ for every $j$ different from $i$. In particular:
\begin{itemize}
	\item $-_{\theta,G\setminus i}=-_{\theta,G}$;
	\item $0_{\theta,G\setminus i}=0_{\theta,G}$;
	\item $+_{\theta,G\setminus i}=+_{\theta,G}$;
	\item $\infty_{\theta,G\setminus i}=\infty_{\theta,G}\setminus i$.
\end{itemize}
\end{thm}
\begin{proof} Consider $i$ in $\partial 0_{\theta,G}$. By Proposition~\ref{preservespm} we need to prove the second and fourth equalities of sets. From Figure~\ref{table} and Proposition~\ref{preservespm} it is clear that $0_{\theta,G}\subseteq 0_{\theta, G\setminus i}$ and $\infty_{\theta,G\setminus i}\subseteq \infty_{\theta,G}$, but it could happen that the intersection $0_{\theta,G\setminus i}\cap\infty_{\theta,G}$ is non-empty.

Assume, by contradiction, that there exists a vertex $j$ in $0_{\theta,G\setminus i}\cap \infty_{\theta,G}$ and let $k\in 0_{\theta,G}$ be a neighbor of $i$.  Consider the graph $G_\epsilon$ obtained from $G$ where the new vertex weight of $i$ is $x-r_i+\epsilon$, with $\epsilon>0$ small. As $i\in 0_{\theta,G\setminus j}\cap\infty_{\theta,G}$ it follows that $i\in +_{\theta,G_\epsilon\setminus j}\cap\infty_{\theta,G_\epsilon}$. By Figure~\ref{table} this implies that $j\in 0_{\theta,G_\epsilon\setminus i}\cap -_{\theta,G_\epsilon}$ and $k\in 0_{\theta, G_\epsilon\setminus i}\cap 0_{\theta,G_\epsilon}$. Since $j\in -_{\theta,G_\epsilon}$ the Figure~\ref{table} implies that $k\in 0_{\theta,G_\epsilon\setminus j}$. Thus, $i\in +_{\theta,G_\epsilon\setminus j}$ is a neighbor of $k\in 0_{\theta,G_\epsilon\setminus j}$ in $G_\epsilon\setminus j$, which is a contradiction by Proposition~\ref{frontier}.
\end{proof}

Recall that a graph $G$ is called {\it $\theta$-critical} if $[n]=0_{\theta,G}$. The $\theta$-critical components of a graph $G$ are the connected components of the induced subgraph in $0_{\theta,G}$. In this context we have the following theorem by Ku and Wong~\cite[p. 3390, Thm. 4.13]{ku2013gallai}.

\begin{thm}(Gallai's lemma analogue by Ku and Wong~\cite{ku2013gallai})\label{gallai} If $G$ is a connected $\theta$-critical graph then $m_\theta(G)=1$.
\end{thm}
\begin{proof} Assume, by contradiction, that $[n]=0_{\theta,G}$ and $m_\theta(G)$ is at least two. Consider a vertex $i$ in $[n]=0_{\theta,G}$. In this case, since $m_\theta(G\setminus i)\geq 1$, the Lemma~\ref{zeroset} implies that $0_{\theta,G\setminus i}$ is non-empty. As $i$ does not belong to $\infty_{\theta,G}$, the Lemma~\ref{internal} implies that the neighbors of $i$ are not in $0_{\theta,G\setminus i}$. Since the graph $G$ is connected, there exists a path from some neighbor of $i$ to a vertex in $0_{\theta,G\setminus i}$. This shows that $\partial 0_{\theta,G\setminus i}$ is non-empty.
	
Let $j$ be a vertex in $\partial 0_{\theta,G\setminus i}$. In particular, by Proposition~\ref{frontier}, $j\in \infty_{\theta,G\setminus i}$. As $j\in \infty_{\theta,G\setminus i}\cap 0_{\theta,G}$ and $i\in 0_{\theta,G}$, we have by Figure~\ref{table} that $i\in\infty_{\theta,G\setminus j}$. This implies by Lemma~\ref{internal} that there exists a neighbor $k$ of $i$ that is in $0_{\theta,G\setminus \{j,i\}}$. But by the Stability Lemma~\ref{stability} applied to $j\in \partial 0_{\theta,G\setminus i}$ it holds that $0_{\theta,G\setminus \{i,j\}}=0_{\theta,G\setminus i}$. Thus the neighbor $k$ of $i$ is in $0_{\theta,G\setminus i}$, which implies by Lemma~\ref{internal} that $i$ is in $\infty_{\theta,G}$, reaching a contradiction. 
\end{proof}

The Theorem~\ref{gallai} leads to the following corollary by Ku and Wong~\cite[p. 3409, Cor. 4.14]{ku2013gallai}.

\begin{cor}(Ku, Wong~\cite{ku2013gallai})\label{multiplicity} The multiplicity of $\theta$ as a zero of $\mu(G)$ is equal to the number of $\theta$-critical components of $G$ minus the number of vertices in $\partial 0_{\theta,G}$. 
\end{cor}
\begin{proof} Note that by the Stability Lemma~\ref{stability} for every subset $S$ of $\partial 0_{\theta, G}$, it holds that $\partial 0_{\theta, G\setminus S} = \partial 0_{\theta, G}\setminus S$. In particular, the $\theta$-critical components of $G\setminus \partial 0_{\theta,G}$ are the $\theta$-critical components of $G$. But in $G\setminus \partial 0_{\theta,G}$ all the $\theta$-critical components are isolated. This implies by Theorem~\ref{gallai} that $m_\theta(G\setminus \partial 0_{\theta,G})$ is equal to the number of $\theta$-critical components of $G$. 

Also note that for each subset $S$ of $\partial 0_{\theta, G}$, removing one vertex of $S$ at a time and using the Stability Lemma~\ref{stability} and Proposition~\ref{frontier}, we get that $m_\theta(G\setminus S)$ is equal to $m_\theta(G)+|S|$. In particular, $m_\theta(G\setminus \partial 0_{\theta,G})$ is equal to  $m_\theta(G)+|\partial 0_{\theta,G}|$, which finishes the proof.
\end{proof}

The Corollary~\ref{multiplicity} implies that if $\theta$ is a zero of $\mu(G)$, then, since $m_\theta(G)\geq 1$, there are more $\theta$-critical components of $G$ than there are vertices in $\partial 0_{\theta,G}$. It turns out that the analogue of item $(c)$ of the Gallai-Edmonds structure theorem~\ref{gallai-edmonds} holds for matching polynomials.

\begin{cor}\label{matched} For every nonempty subset $S$ of $\partial 0_{\theta,G}$ there are at least $|S|+1$ $\theta$-critical components of $G$ that are connected to a vertex in $S$. 
\end{cor}
\begin{proof} By the Stability Lemma~\ref{stability} we can restrict ourselves to the graph $G'$ obtained from $G$ by first deleting all the vertices in $\partial 0_{\theta,G}\setminus S$ and then deleting all the isolated $\theta$-critical components. Observe that $\partial 0_{\theta,G'}=S$ and all the $\theta$-critical components of $G'$ are $\theta$-critical components of $G$. Note that since $S$ is nonempty, $G'$ has at least one $\theta$-critical component. Since $m_\theta(G')\geq 1$ and $\partial 0_{\theta,G'}=S$, the Corollary~\ref{multiplicity} then implies that there at least $|S|+1$ $\theta$-critical components in $G'$. But there are no isolated $\theta$-critical components in $G'$, so all of the $|S|+1$ $\theta$-critical components are connected to a vertex in $S$.
\end{proof}

Using Corollary~\ref{matched} and the path tree we can give a conceptual explanation for why the Stability Lemma~\ref{stability} is true. Let $i$ and  $j\in\partial 0_{\theta,G}$ be two distinct vertices of the graph $G$. Consider the tree continued fraction $\alpha_i(T^i_G)(\theta)$. 
Observe that for every path $c:i=i_1\to i_k=j$ it holds $\alpha_j(G\setminus \{i_1,\dots,i_{k-1}\})(\theta)=\infty$. This means that along the tree continued fraction $\alpha_i(T^i_G)(\theta)$, the vertex $j$ always corresponds to a node with an infinity, and so it can be disregarded.

In order to see this, note that if the path $c$ does not go through $0_{\theta,G}$, then by Figure~\ref{table} the original $\theta$-critical components are unaffected in the graph $G\setminus \{i_1,\dots,i_{k-1}\}$. It follows that there is a remaining $\theta$-critical component connected to $j$ which guarantees $\alpha_j(G\setminus \{i_1,\dots,i_{k-1}\})(\theta)=\infty$. Now, if the path $c$ goes through $0_{\theta,G}$, then the Corollary~\ref{matched} guarantees that there is also a remaining $\theta$-critical component connected to $j$ in $G\setminus \{i_1,\dots,i_{k-1}\}$.

With the same reasoning, it is clear that the following version of the Stability Lemma~\ref{stability} is also true.
 
\begin{cor}(Stability lemma II) Let $G$ be a graph with two distinct vertices $i$ and $j\in \partial 0_{\theta,G}$. Consider the graph $G'$ obtained from $G$ where the weights $r_j$ and $\lambda_{jk}\leq 0$, for all $k\neq j$, are modified. Assume that for every subset $S$ of $\partial 0_{\theta,G}$ there are at least $|S|+1$ $\theta$-critical components of $G$ that are connected to a vertex in $S$ in the graph $G'$. In this case, $\alpha_i(G')(\theta)=\alpha_i(G)(\theta)$ for every vertex $i$.
\end{cor}


\section{Applications}\label{applications}

\subsection{Lower Bound for Largest Zero of $\mu(G)$}

Using the techniques from the last section, we were also able to easily obtain a lower bound to the greatest zero of a matching polynomial in the same spirit as Corollary~\ref{Heilmann-Lieb-real} and the Heilmann-Lieb Theorem~\ref{Heilmann-Lieb}. 

Observe that the classical Heilmann-Leib Theorem~\ref{Heilmann-Lieb} only gives an upper bound on the size of the greatest zero of a matching polynomial. For example, if $G$ is a $d$-regular graph with $d\geq 3$, then the Corollary~\ref{Heilmann-Lieb-real} gives an upper bound of $2\sqrt{d-1}$ for the largest zero of the matching polynomial of $G$. The result of this section implies that this largest zero is also bigger than $\sqrt{d}$.

Our proof proceeds by first showing, using the results from the last section, that the greatest zero of a matching polynomial is simple, a result that is already found in the work of Godsil and Gutman~\cite[p. 143]{godsil1981theory}. With this we show that the largest zero of a matching polynomial decreases as the weights on the edges increase. Using this last fact, we easily obtain the promised lower bound. An analogous strategy can also be used to establish an upper bound for the smallest zero of the matching polynomial.

The second part of the next lemma appears in Godsil and Gutman's work~\cite[p. 143]{godsil1981theory}.

\begin{lem}\label{simple} Let $G$ be a connected graph. If $\theta$ is the largest or smallest zero of $\mu(G)$, then $G$ is $\theta$-critical. In particular, $\theta$ is a simple zero of $\mu(G)$.
\end{lem}
\begin{proof} Let $\theta$ be the smallest zero of $\mu(G)$. By Proposition~\ref{zeroset} we know that $0_{\theta,G}$ is non-empty. Observe that, if $x$ is smaller than $\theta$, then $[n]=-_{x,G}$. This implies that at time $\theta$ all the vertices are in $-_{\theta,G}\sqcup 0_{\theta,G}$. But $G$ is connected and by Proposition~\ref{frontier} it holds $\partial 0_{\theta,G}\subseteq \infty_{\theta,G}$, so it must be $[n]=0_{\theta,G}$. This implies, by Theorem~\ref{gallai} that $m_\theta(G)=1$, so $\theta$ is a simple zero. For the largest zero of $\mu(G)$ the proof is analogous.
\end{proof}

Given a graph $G$ denote by $z_G$ the largest zero of its matching polynomial.

\begin{lem}\label{largestzero} Let $G$ be a connected graph and consider the  graph $G'$ obtained from $G$ where the edge $ij$ receives a new weight $\lambda_{ij}<\lambda_{ij}'\leq 0$. In this case, $z_G>z_{G'}$.
\end{lem}
\begin{proof} The Lemma~\ref{simple} shows that $G$ is $z_G$-critical. Using the interlacing of Corollary~\ref{interlacing} this implies that the largest zero of $\mu(G\setminus \{i,j\})$ is smaller than $z_G$. It follows that $\mu(G)(x)\geq 0$ and $\mu(G\setminus \{i,j\})(x)>0$ for $x\geq z_G$. As a consequence, $\mu(G')(x)=\mu(G)(x)-(\lambda_{ij}-\lambda_{ij}')\mu(G\setminus \{i,j\})(x)>0$, for $x\geq z_G$. This shows that the largest zero of $\mu(G')$ is smaller than the largest zero of $\mu(G)$. 
\end{proof}

\begin{lem} Let $G$ be a connected graph with at least three vertices and $r_i=\ds\max_jr_j$. In this case,
	
\[r_i<z^*\leq z_G<r_i+2\sqrt{\ds\max_j\,\max_{\mathclap{\substack{A\subseteq [n]\setminus j\\ |A|=n-2}}}\quad\sum_{k\in A}-\lambda_{jk}},\,\,\text{where } z^*=r_i+\ds\sum_{j\neq i}\dfrac{-\lambda_{ij}}{z^*-r_j}.\]
	
In particular, if $r_j$ is zero for every $j$, then
	
\[\sqrt{\ds\max_j\sum_{k\neq j}-\lambda_{jk}}\leq z_G<2\sqrt{\ds\max_j\,\max_{\mathclap{\substack{A\subseteq [n]\setminus j\\ |A|=n-2}}}\quad\sum_{k\in A}-\lambda_{jk}}.\]
\end{lem}
\begin{proof} The upper bound for $z_G$ comes from Corollary~\ref{Heilmann-Lieb-real} and the fact that $G$ has at least three vertices. For the lower bound consider the graph $G'$ obtained from $G$ where all the edges that are not adjacent to $i$ are set to zero. By Lemma~\ref{largestzero} the largest zero of $\mu(G')$, denoted by $z^*$, is less than or equal to $z_G$. As $i$ is not an isolated vertex in $G'$ the Lemma~\ref{largestzero} implies that $z^*$ is bigger than $r_i$. Finally, observe that,
	
\[\mu(G')(x)=\ds\prod_j(x-r_j)+\ds\sum_{j\neq i}\lambda_{ij}\ds\prod_{k\neq i,j}(x-r_k)\implies\]

\[\ds\prod_j(z^*-r_j)=\ds\sum_{j\neq i}-\lambda_{ij}\ds\prod_{k\neq i,j}(z^*-r_k)\implies z^*-r_i=\ds\sum_{j\neq i}\dfrac{-\lambda_{ij}}{z^*-r_j}.\]
\end{proof}


\section*{Acknowledgements}

This work is partially based on my Ph.D. Thesis at IMPA, Brazil. The author acknowledges the support from CAPES-Brazil scholarship grant.

\IfFileExists{references.bib}
{\bibliography{references}}
{\bibliography{../references}}
	
	
\end{document}